\documentclass[review,3p,12pt, nopreprintline]{elsarticle}
\usepackage{lineno,hyperref}
\usepackage{lipsum}
\journal{ }

\def\R{{\mathbb R}}
\def\N{{\mathbb N}}
\def\dee{{\rm d}}

\def\:{{\colon}}

\usepackage{graphicx,epsfig,psfrag}
\usepackage{amssymb}
\usepackage{xcolor}
\usepackage{amsthm}
\usepackage{mathtools}
\usepackage{mathrsfs}
\usepackage{calrsfs}
\usepackage{mathscinet}
\usepackage{hyperref}
\usepackage{amsmath}
\usepackage{enumerate}
\numberwithin{equation}{section}
\newtheorem{theorem}{Theorem}[section]
\newtheorem{corollary}[theorem]{Corollary}

\newdefinition{remark}{Remark}[section]
\newdefinition{definition}{Definition}[section]
\newdefinition{example}{Example}[section]

\newcommand{\be}{\begin{equation}}
\newcommand{\ee}{\end{equation}}
\newcommand{\ba}{\begin{eqnarray}}
\newcommand{\bs}{\begin{eqnarray*}}
\newcommand{\ea}{\end{eqnarray}}
\newcommand{\es}{\end{eqnarray*}}
\newcommand{\bi}{\begin{itemize}}
\newcommand{\ei}{\end{itemize}}
\newcommand{\A}{{\alpha}}
\newcommand{\eps}{{\varepsilon}}
\newcommand{\lam}{{\lambda}}

\newcommand{\Lap}{\Delta}
\newcommand{\fLap}{\Delta_{\alpha}}
\newcommand{\B}{{\beta}}

\newcommand{\Is}{\int_{\R^n }}

\newcommand{\fp}{\tilde{f}}
\newcommand{\fs}{{F}}
\newcommand{\fps}{{f}}
\newcommand{\pc}{{p_{\alpha}}}
\newcommand{\F}{{\mathscr F}}

\allowdisplaybreaks
\begin{document}


\begin{frontmatter}

\author[RL]{R.\ Laister}
\ead{Robert.Laister@uwe.ac.uk}
\address[RL]{Department of Engineering Design and Mathematics, \\ University of the West of England, Bristol BS16 1QY, UK.}
\author[warsaw]{M.\ Sier{\.z}\polhk{e}ga}
\ead{M.Sierzega@mimuw.edu.pl}
\address[warsaw]{Faculty of Mathematics, Informatics and Mechanics,  University of Warsaw,\\  Banacha 2, 02-097 Warsaw, Poland.}
%

\title{A Blow-up Dichotomy for Semilinear Fractional  Heat Equations }


\begin{abstract}
We  derive a blow-up dichotomy  for positive solutions of  fractional  semilinear heat equations on the whole space.
That is, within a certain class of convex  source terms, we establish a necessary and sufficient
 condition on the source for all positive  solutions to become unbounded in finite time.
  Moreover, we show that this condition is equivalent to blow-up of all positive solutions of a closely-related scalar ordinary differential equation.

\end{abstract}


\end{frontmatter}


\section{Introduction}
In this paper we investigate the local and global existence  properties of positive solutions of  fractional semilinear heat equations of the form
\begin{equation}\label{nhe}
u_t=\fLap u+f(u),\qquad u(0)=\phi\in L^{\infty}(\R^n),
\end{equation}
where  $\fLap =-\left(-\Lap \right)^{{\alpha}/{2}}$  denotes   the fractional Laplacian operator  with $0<\alpha\le 2$ 
and $f$ satisfies the monotonicity condition
\bi
\item[{\bf {\bf  (M)}}] $f\colon [0,\infty)\to [0,\infty)$ is locally Lipschitz continuous, non-decreasing and $f(0)=0$.
\ei
We present a new dichotomy result  for convex nonlinearities $f$ satisfying the  ODE blow-up criterion
\bi
\item[{\bf  (B)}] $\displaystyle{\int_{1}^{\infty}\frac{1}{f(u)}\, \dee  u <\infty ,}$
\ei
together with  an additional, technical assumption  {\bf  (S)}  (see Section 2).  Specifically, for this class of nonlinearities, we show that all positive solutions of (\ref{nhe}) blow-up in $L^{\infty}(\R^n)$ in finite time {\em if and only if} 
\be
\int_{0^+}\frac{f(u)}{u^{2+\alpha /n}}\, \dee  u =\infty  .\label{intzero}
\ee
Furthermore, we  establish an equivalence between finite time blow-up of all positive  solutions of \eqref{nhe} and  finite time blow-up of all positive solutions of 
 the  scalar, non-autonomous ODE
\be
x'=f(x)-\left(\frac{n}{\alpha t}\right)x,\qquad x(t_0)=x_0>0.\label{introODE}
\ee
To the best of our knowledge  this kind of blow-up equivalence,  between the PDE \eqref{nhe} and a scalar ODE such as (\ref{introODE}), has not been established before. 

We will refer to the  phenomenon of blow-up in  finite time of  all non-negative, non-trivial   solutions of (\ref{nhe}) simply as the `{\it blow-up property}'. We will also 
identify the phrase `non-negative, non-trivial solution' synonomously with `positive solution'.

For the case of classical diffusion ($\alpha =2$) it has long been known  that for $f$ convex   and sufficiently large  initial data $\phi$,  blow-up in \eqref{nhe} occurs; 
  see  \cite[Theorem 17.1]{QS2} for bounded domains and  the whole space alike. The central question then was whether  diffusion  could prevent blow-up  for initial  data sufficiently small.
For general continuous sources $f$, this problem is highly non-trival  and remains open. However, under further restrictions on the form of the nonlinearity there has been significant progress, for example when  $f$ is the power law nonlinearity $f(u)=u^{p}$. In \cite{Fuj66}  a  threshold phenomenon was established, whereby the (Fujita) critical exponent, given by $p_F=1+{2 }/{n}$, separated two regimes:
 for $1<p<p_F $  \eqref{nhe} has the blow-up property, whereas for $p> p_F $ it is possible to find  \emph{small} initial
  conditions $\phi$ evolving into global-in-time solutions. Non-existence of positive global solutions in the delicate critical case $p=p_F $ was later established in \cite{Hay}
   for the case $n\le 2$ and subsequently by \cite{Sug} for all $n\ge 1$. Thus was obtained the first blow-up dichotomy for  \eqref{nhe}{\color{blue}:} in the special case 
   $f(u)=u^p$ and $\alpha =2$,  \eqref{nhe} has the blow-up property if and only if $1<p\le p_F $.  Some slight generalisations can also be found in  \cite{Fuj70,FW}.
In fact the result obtained in  \cite[Theorem]{Sug}  extended previous work  on blow-up in two important ways:  firstly to convex sources terms $f$, and secondly to {\em fractional} diffusion operators.  Specifically, it was shown  for convex $f$  satisfying {\bf  (M)} and {\bf  (B)}, that if 
\be
\lim_{u\to 0}\frac{f(u)}{u^{\pc}}>0, \qquad {\text {where}}\quad \pc :=1+\frac{\A}{n},\label{eqn:pc}
\ee  
then  (\ref{nhe}) has the blow-up property. In fact it is easy to see from the proofs in \cite{Sug} that (\ref{eqn:pc}) need only hold in the  limit inferior sense. 
 There are many other works  which consider the global and blow-up solution properties of nonlinear  fractional diffusion equations,
  all assuming either a power law nonlinearity or a convex one bounded below by a power law near zero as in (\ref{eqn:pc}); see   e.g., \cite{BLW,HKN,IKK,MQ,NS}.

 Subsequently it was shown  in \cite{KST}, in the special case of classical diffusion ($\A =2$), that  condition (\ref{eqn:pc})  is not necessary in order  for (\ref{nhe}) to  have the blow-up property;  this can be seen via  the example in \cite[Section 5]{KST} where a logarithmic-type correction of the critical Fujita case is considered.  In that work it was shown
  (\cite[Theorem~4.1]{KST})  that 
  (\ref{nhe})  has the blow-up property if $f$ is continuous on $[0,\infty )$, positive  on $(0,\infty )$, $f(0)=0$,  {\bf  (B)} and  (\ref{intzero}) hold and $f$ satisfies 
  a further technical condition (labelled (B.3) in \cite{KST}). In particular, neither monotonicity nor  convexity of $f$ were required. On the other hand, 
  this blow-up result is restricted to the  case $\A =2$ and their technical condition  (B.3) still  imposes a certain logarithmic scaling bound near zero; see Section 4 later on for more details.  Conversely, when (\ref{intzero}) fails the authors in \cite{KST}  go on to prove a global-in-time existence result for small initail data, subject to stronger regularity and monotonicity conditions on $f$.  
   
An important aspect of this paper is that we demonstrate (via an explicit construction) that, for all $\A\in (0,2]$,   there exist    monotone, convex  $f$ for which 
(\ref{nhe})  has the blow-up property, but for which the results in  \cite{Sug} and \cite{KST}) do not apply.
  
  The remainder of the paper is organised as follows. In  Section 2 we prove that, for a suitable class of sources $f$, (\ref{intzero}) is sufficient for the ODE (\ref{introODE}) to have the blow-up property.  In  Section 3 we show for this class that if the ODE (\ref{introODE}) possesses the blow-up property then so too does (\ref{nhe}). In Section 4 we present a construction 
  which demonstrates  that our  assumption {\bf  (S)} (stated below)  is strictly weaker than (B.3) of \cite{KST} in the case $\A =2$. We then establish in Section 5  the necessity of  (\ref{intzero})  for  (\ref{nhe}) to have the blow-up property and conclude with some remarks in Section 6.


\section{Blow-up of a Related  ODE}


Here we consider the  blow-up properties of the  non-autonomous ODE
\be
x'=f(x)-\left(\frac{n}{\alpha t}\right)x,\qquad x(t_0)=x_0>0,\quad t_0>0.\label{ODE}
\ee

\begin{definition}
Suppose $f$ satisfies {\bf  (M)}. We say that the  ODE (\ref{ODE})  has the  \emph{blow-up property} if the  solution of (\ref{ODE}) blows-up  in finite time for every  $x_0>0$ and $t_0>0$.

\label{def:odegep}
\end{definition}

We now introduce some further hypotheses:
\bi
\item[{\bf {\bf  (C)}}] $f$ is  convex on $(0,\infty )$;
\item[{\bf  (S)}]  there exist $c_0, \mu_0>0$ and $g\colon (\mu_0,\infty )\to (0,\infty )$ such that $\int_{\mu_0}^{\infty}1/g(s)\, \dee  s  <\infty$ 
and 
\bs 
f(\lam\mu )\ge g(\mu )f(\lam )\quad \text{for all}\quad \mu \ge \mu_0\quad \text{and}\quad \lam\mu\in (0,c_0 ).\es
\ei

\begin{remark}
\bi
\item[]
\item[(i)] If   $f(u)/u^{p}$ is non-decreasing near zero, on $(0,c_0)$ say, for some   $p >1$, then  {\bf  (S)} holds with  $g(\mu )=\mu^{p}$. To see this,
observe that for any choice of $\mu_0 \ge 1$ we have, for $\mu\ge\mu_0$ and  $\lam\mu\in (0,c_0 )$,
\bs
\frac{f(\lam\mu )}{(\lam\mu)^p}\ge \frac{f(\lam)}{\lam^p}.
\es
Hence $f(\lam\mu )\ge { \mu^{p}}f(\lam)$ for all  $\mu\ge\mu_0$ and  $\lam\mu\in (0,c_0 )$.  

  The particular,  homogeneous, Fujita-critical case where  $g(\mu )=\mu^{p}$ and 
$p =p_F $ was considered in \cite[(B.3)]{KST} on a {\em strictly larger} $\lam$-$\mu$ region than appears in {\bf  (S)}; i.e., the condition imposed upon $f$ in \cite{KST}  is a more restrictive one than that in {\bf  (S)}.

We mention also that a  condition such as  $f(u)/u^{p}$ being non-decreasing  was  used in \cite{BaC}, although there the condition at  infinity was relevant rather than near zero.


\item[(ii)] It is  easy to verify that if   $ 0\neq f\in C^1$  satisfies {\bf  (M)} and {\bf  (C)} and the condition
\be
\liminf_{u\to 0}\frac{uf^\prime (u)}{f(u)}>1,\label{eq:liminf}
\ee
then  there exists a $p >1$ such $f(u)/u^{p}$ is non-decreasing near zero. Consequently $f$ satisfies {\bf  (S)} by  (i) above.
 Note that for  $ 0\neq f\in C^1$  satisfying {\bf  (M)} and {\bf  (C)}, we always have  $uf^\prime (u)/f(u)\ge 1$ for all $u> 0$.
\label{rem:BUP}
\ei
\end{remark}


\begin{theorem}\label{thm:BUPode}
Suppose $f$ satisfies {\bf  (M)},  {\bf  (C)}, {\bf  (B)} and {\bf  (S)}. If
\be
\int_{0^+}\frac{f(u)}{u^{2+\alpha /n}}\, \dee  u =\infty  ,\label{B2inf}
\ee
then the ODE  (\ref{ODE}) has the   blow-up property.
\end{theorem}

\begin{proof}
Suppose, for contradiction, that there exists a global solution $x(t)$ of (\ref{ODE}). By ODE uniqueness it is clear that the solution of (\ref{ODE}) is positive for all $t\ge t_0 $. By {\bf  (M)} and (\ref{B2inf}),  $f>0$ on $(0,\infty)$ and  by {\bf  (C)},   $L (u):=f(u)/u$ is non-decreasing and positive for $u>0$.

Suppose first that $x$ is bounded away from zero, i.e., there exists  $\eps >0$ such that $x(t)\ge \eps$ for all $t\ge t_0 $.  By monotonicity of $L$,  $L (x(t))\ge L(\eps)>0$ for all $t\ge t_0  $.  Hence there exists $t_1> t_0 $ such that
\bs
1-\left(\frac{n}{\alpha t}\right)\frac{1}{L(x(t))}\ge 1-\frac{n}{\alpha L(\eps)t}\ge 1/2 \es
for all $t\ge t_1$. For such $t$ we  have
\bs
x'=f(x)\left( 1-\left(\frac{n}{\alpha t}\right)\frac{1}{L(x(t))}\right) \ge  f(x)/2,
\es
and so by {\bf  (B)} $x$  blows up in finite time, a contradiction.

Now suppose  that $x$ does not remain bounded away from zero. We then claim that $x^{\prime}(t)\le 0$ for all $t\ge t_0 $.
For suppose this is not the case, so that  there exists $t_2\ge t_0 $  such that $x^{\prime}(t_2)> 0$. Since $x$ is $C^1$ and not bounded away from zero, there exists  $t_3>t_2$ such that $x^{\prime}(t)> 0$ for all $t\in [t_2,t_3)$ and  $x^{\prime}(t_3)= 0$. Clearly $x(t_2)<x(t_3)$ and so by the monotonicity of
$L$ we have
\bs
0=\frac{x^{\prime}(t_3)}{x(t_3)}= L (x(t_3))-\frac{n}{\alpha t_3}
> {L(x(t_2))}-\frac{n}{\alpha t_2}
=\frac{x^{\prime}(t_2)}{x(t_2)}>0,
\es
which is clearly false. Hence $x^{\prime}(t)\le 0$ for all $t\ge t_0 $ as claimed. It follows that $x(t)$ is non-increasing and $x(t)\to 0$ as ${t\to\infty}$.

Now set $y(t)= t^{n/\alpha }x(t)$ so that $y$ satisfies the ODE
\be
y^\prime =yL\left(yt^{-n/\alpha }\right),\qquad y(t_0)=t_0^{n/\alpha }x_0=\colon  y_0>0.
\label{eq:yODE}
\ee
By (\ref{eq:yODE}), $y$ is clearly increasing and so
\bs
y(t)&= &\exp{\left(  \int_{t_0}^{t} L\left(y(s)s^{-n/\alpha }\right)\,\dee s  \right)} \ge \exp{\left(  \int_{t_0}^{t} L\left(y_0s^{-n/\alpha }\right)\,\dee s  \right)}\\
&= &\exp \left(  \frac{\alpha y_0^{\alpha /n}}{n}\int_{y_0t^{-n/\alpha }}^{x_0} \frac{f(u)}{u^{2+\alpha /n}}\, \dee  u  \right)\to  \infty 
\es
{as} $t\to\infty$, by (\ref{B2inf}).  For $\tau >t_0$ sufficiently large  we can ensure that  $y(t)\ge \mu_0$ and $ t^{-n/\alpha }y(t)=x(t)\le c_0 $ for all $t\ge \tau$.
For such $t$ it follows from {\bf  (S)} that
\bs
y^\prime &=& t^{n/\alpha }f\left(t^{-n/\alpha }y\right)\ge  t^{n/\alpha }f\left(t^{-n/\alpha }\right)g(y)
\es
and so
\bs \int_{y(\tau)}^{y(t)}\frac{\dee y}{g(y)} &\ge & \int_{\tau}^{t} L\left(s^{-n/\alpha }\right)\,\dee s=  \frac{\alpha }{n}\int_{t^{-n/\alpha }}^{\tau^{-n/\alpha }} \frac{f(u)}{u^{2+\alpha /n}}\,\dee u.
\es
Letting $t\to\infty$  and  using  (\ref{B2inf})  we again obtain a contradiction, on  recalling  the integrability of  $1/g$ in  {\bf  (S)}.
\end{proof}


\section{Blow-up of the  PDE}

In this section we show that blow-up of the ODE (\ref{ODE}) implies blow-up of the PDE (\ref{nhe}). We denote by $\left\{ S_{\alpha}(t)\right\}_{t\ge 0}$  the fractional heat semigroup on $L^q(\R^n  )$ ($q\ge 1$) generated by $-\fLap$ on  $\R^n$ with  the explicit representation formula
 \be
 [S_{\alpha}(t)\phi](x)=\Is K_{\alpha}(x-y,t)\phi (y)\, \dee  y,\qquad \phi\in L^q(\R^n  ),\label{eq:heatsg}
 \ee
where $K_{\alpha}$ is the (positive) fractional heat kernel. As is commonplace in the study of semilinear problems, we may  then study (\ref{nhe}) via the variation of constants formula
\begin{equation}
u(t)={\mathscr F}(u;\phi)\coloneqq S_{\alpha}(t)\phi+\int_0^t S_{\alpha}(t-s)f(u(s))\, \dee  s.\label{eq:VoC}
\end{equation}

It is  well known  that for 
 any  non-negative initial condition $\phi\in L^{\infty}(\R^n  )\ $ there is a $T_{\phi}>0$ such that  (\ref{nhe})  has a unique non-negative  solution $u$
 which is bounded on $\R^n \times [0,T]$ for any $T\in (0,T_\phi)$, such that if $ T_{\phi} <\infty$ then $\| u(t)\|_{\infty}\to \infty$ as $t\to T_{\phi}$. If $T_{\phi} =\infty$ then we say that $u$ is a  {\em{global solution}} of (\ref{nhe}).

\begin{definition}

Suppose $f$ satisfies {\bf  (M)}.  We say that  the PDE (\ref{nhe}) has the  \emph{blow-up property} if for every non-trivial, non-negative $\phi\in L^{\infty}(\R^n  )$ we have 
$ T_{\phi} <\infty$.
\label{def:pdegep}
\end{definition}

\begin{theorem}
Suppose that $f$ satisfies {\bf  (M)} and {\bf  (C)}. If  the ODE (\ref{ODE}) has the blow-up property then the PDE (\ref{nhe}) has the blow-up property.
\label{thm:blowup}
\end{theorem}

\begin{proof}
We proceed  as in the proof of the main theorem in \cite[Section 4]{Sug}. We briefly outline the initial steps of that proof for the reader's convenience.

Suppose, for contradiction, that $u$ is a  non-negative,  global solution of (\ref{nhe}). Then  $u$  satisfies the integral equation
\be
u(x,t)=\int_{\R^n} K_{\alpha}(x-y,t)\phi(y)\, \dee y+\int_0^t \int_{\R^n} K_{\alpha}(x-y,t-s)f(u(y,s))\, \dee y\dee  s.\label{eq:integral}\ee
Clearly $u>0$ for all $t>0$ and so, by translating in time if necessary, we may assume without loss of generality that  $\phi >0$.


   Using the integral formulation (\ref{eq:integral}), positivity of the solution and standard properties of  $K_{\alpha}$, one can then  show that there exist constants
    $c>0$, $\tau_0>0$ and $t_0>0$ such that $u(x,t_0)\ge cK_{\alpha}(x,\tau_0 )$ for all $x\in\R^n$
    (see \cite[p.\,48]{Sug}). It follows that
\bs
u(x,t+t_0 )&= & \int_{\R^n}K_{\alpha}\left(x-y,t\right){  u(y,t_0)}\,\dee y\\
&& \qquad+\int_0^t\int_{\R^n} K_{\alpha}(x-y,t-s)f(u(y,s+t_0))\,\dee y\dee s\\
&\ge & {  c\,K_{\alpha}(x,t+\tau_0)}+\int_0^t\int_{\R^n} K_{\alpha}(x-y,t-s)f(u(y,s+t_0))\,\dee y\dee s.
\es
Setting $v(x,t)=u(x,t+t_0 )$ yields
\be
v(x,t)=   c\,K_{\alpha}(x,t+\tau_0)+\int_{t_0}^t\int_{\R^n} K_{\alpha}(x-y,t-s)f(v(y,s))\,\dee y\dee s.\label{eq:intv}
\ee
Clearly  $v(t)\in L^{\infty}(\R^n)$ for all $t> 0$ since $u$ is assumed to be in $L^{\infty}(\R^n)$ for all $t> 0$. 
Now set
$$z(t)=\int_{\R^n}K_{\alpha}\left(x,t\right)v(x,t)\,\dee x.$$
Evidently  $z(t)$ is positive and finite for all $t> 0$.  Multiplying (\ref{eq:intv})  by $K_{\alpha}(x,t)$, integrating over $\R^n$ and using the semigroup property of $K_\alpha$,  gives
\bs
z(t)&= & k(2t+\tau_0)^{-n/\alpha }+\int_{t_0}^t\int_{\R^n} K_{\alpha}(y,2t-s)f(v(y,s))\,\dee y\dee s.
\es
where  $k=k(n,\alpha ,c)$ is a positive constant. Now using  the scaling property
\bs
K_{\alpha}(x,t)=t^{-n/\alpha }K_{\alpha}(t^{-1/\alpha }x,1)
\es 
of $K_{\alpha}$ (see e.g., \cite[ p.\,46-47]{Sug}) and the fact that $K_{\alpha}(x,t)$
is decreasing in $|x|$, we have for $s\le t$,
\bs
K_{\alpha}(y,2t-s)&=&\left(2t-s\right)^{-n/\alpha }K_{\alpha}\left(\left(2t-s\right)^{-1/\alpha}y,1\right)\\
&\ge & \left(2t-s\right)^{-n/\alpha }K_{\alpha}\left(s^{-1/\alpha}y,1\right)\\
&= &\left(\frac{2t-s}{s}\right)^{-n/\alpha }K_{\alpha}\left(y,s\right)\\
&\ge & 2^{-n/\alpha }\left({t}/{s}\right)^{-n/\alpha }K_{\alpha}\left(y,s\right).
\es
Hence, by Jensen's inequality,
\ba
z(t)
&\ge & k(2t+\tau_0)^{-n/\alpha } +2^{-n/\alpha }\int_{t_0}^t\left({t}/{s}\right)^{-n/\alpha }\int_{\R^n} K_{\alpha}\left(y,s\right)f(v(y,s))\,\dee y\dee s\nonumber \\
&\ge & k(2t+\tau_0)^{-n/\alpha } +2^{-n/\alpha }\int_{t_0}^t \left({t}/{s}\right)^{-n/\alpha }f(z(s))\,\dee s \nonumber\\
&\ge & \kappa {  t^{-n/\alpha }} +\kappa \int_{t_0}^t {  \left(t/s\right)^{-n/\alpha }}f(z(s))\,\dee s \label{departure}
\ea
 for all $t\ge t_1$,    $\kappa <2^{-n/\alpha }\min\{1,k\}$  and $t_1>t_0$ sufficiently large. Here we point out that (\ref{departure}) is a 
departure from the form used in \cite[Equation (4.4)]{Sug} and is the reason for our introduction and analysis of the auxiliary ODE (\ref{ODE}).

 It now follows from (\ref{departure}) that for $t>t_1$, $z$ is  a supersolution of the ODE
\bs
w^\prime = \kappa f(w)-\left(\frac{n}{\alpha t}\right)w.
\es
By rescaling time ($t\mapsto \kappa t$) we see  that $z(t)\ge x(\kappa t)$, where $x$ is the solution of the ODE
  $$x^\prime = f(x)-\left(\frac{n}{\alpha t}\right)x, \qquad x(\kappa t_1)=z(t_1)>0.$$
  By assumption  $x$ (and hence $z$) blows-up in finite time, yielding the required contradiction to our earlier statement that $z(t)$ is finite for all $t>0$.
  \end{proof}

By  Theorem~\ref{thm:BUPode} and  Theorem~\ref{thm:blowup} we obtain the following blow-up result for (\ref{nhe}).

\begin{corollary}\label{cor:BUPpde}
Suppose $f$ satisfies {\bf  (M)},  {\bf  (C)}, {\bf  (B)} and {\bf  (S)}. If
\bs
\int_{0^+}\frac{f(u)}{u^{2+\alpha /n}}\, \dee  u =\infty  ,
\es
then the PDE  (\ref{nhe}) has the   blow-up property.
\end{corollary}


\section{A Distinguishing Example}

In this section we present an example of a function $f$ which satisfies the hypotheses of  Corollary~\ref{cor:BUPpde}, but not those of 
 \cite[Theorem 4.1]{KST} (for $\alpha =2$)   nor \cite[Theorem]{Sug}.  Firstly, the $f$ we construct below  has the property 
 \be
\liminf_{u\to 0}\frac{f(u)}{u^{\pc}}=0,
\label{liminf}
\ee 
so that the requirement (\ref{eqn:pc}) (labelled (F.2) in  \cite{Sug}) does not hold. 

In \cite[Theorem 4.1]{KST} the authors use a similar but more restrictive version of {\bf  (S)} (there denoted by assumption (B3)) to establish blow-up of the PDE
 (\ref{nhe}) when  $\alpha =2$. The authors assume that there exists  $c_0\in (0,1]$   such that
\be
f(\lam\mu )\ge c_0\mu^{p_F }f(\lam )\quad \text{for all}\quad 0<\lam\le \mu ,\quad  \lam\in (0,c_0 )\quad \text{and}\quad \lam\mu\in (0,c_0 ).\label{B3}
\ee
Hence the r\^ole  of $g$ in {\bf  (S)} is played there by the power law nonlinearity $g(\mu )=\mu^{p_F }$ (recall Remark~\ref{rem:BUP}). In fact the {\it homogeneity} of this power law function is crucial in  the proof of \cite[Theorem 2.1]{KST} (see also \cite[Theorem 3.5]{KST}) on which \cite[Theorem 4.1]{KST} relies.  Indeed, the iterative  blow-up procedure used in the proof of \cite[Theorem 2.1]{KST}  utilizes in an essential way certain scaling identities relating the exponential function in the Gaussian heat kernel $K_2$ and the power law.
Furthermore, {\bf  (S)} is not required to hold for arbitrarily small $\mu$ which is an essential requirement in the proof of \cite[Theorem 2.1]{KST}. On the other hand we impose the stronger convexity assumption in {\bf  (C)}, absent in \cite{KST}.

Now observe, upon taking $\lam =\mu$,  that any $f$ satisfying  (\ref{B3})   necessarily satisfies 
\be
\liminf_{\lam\to 0}\frac{f (\lam^2)}{\lam^{p_F} f(\lam)}>0.
\label{diag}
\ee
(Note that this in turn imposes upon $f$ a kind of logarithmic scaling bound, as emerges in the proof of \cite[Lemma 3.6]{KST}).  
The $f$ we construct below will satisfy  the hypotheses of  Corollary~\ref{cor:BUPpde} (for any $\alpha\in (0,2]$), but {\em not}  (\ref{diag}) (for $\alpha =2$).

Let $\alpha\in (0,2]$.


{\it Step 1}. 
 Define the monotonically decreasing sequences
\be
\sigma_{i}:=e^{-i^2},\qquad u_i:=e^{-e^{i^2}}, \qquad i\in \N .\label{seq}
\ee
Let $I_i$ denote the interval $I_i=[u_{i+1},u_i)$ and define 
$$\fp (u)=\sigma_{i}u^\pc , \qquad u\in I_i, \qquad i\in \N$$
with $\fp (0)=0$.   Notice that $\fp (u)=u^\pc/\ln (1/u_i)$ on $I_i$, so that $\fp$ models a  logarithmic correction to the critical power law case 
on a  sequence of vanishingly small intervals near zero, to be compared with the example of \cite[Section 5]{KST}.
 It is clear that $\fp $ is  non-decreasing on $[0,\delta ]$ (where $\delta >0$ can be chosen as small as desired later on) with discontinuities at $u=u_i$.

\vspace{1cm}
{\it Step 2}. We now modify $\fp$ to create a function $\fps$ satisfying {\bf  (M)}, {\bf  (C)}  and {\bf  (B)}.

Fix   $p$ and $\theta >1$ such that 
\be
1<p<\pc ,\qquad \frac{\pc}{\pc -1}<\theta <\frac{p}{p-1}\label{params}
\ee
and set $v_i=\theta u_{i+1}$. It is easily verified  that $\theta u_{i+1}<u_i$ for all $i$ sufficiently large, and so for all such $i$
$$ u_{i+1}<v_i<u_i.$$
Now set $J_i= [u_{i+1},v_i)$ and $M_i=[v_i,u_i)$ (so that $I_i$ is the disjoint union of $J_i$ and  $M_i$) and  define
\bs
\fps (u)=\left\{
\begin{array}{ll}
b_iu -a_i,  & u\in J_i, \\
\sigma_{i}u^\pc , & u\in M_i 
\end{array}\right.\es
for $i$ large, with $\fps (0)=0$.   Note that $f(v_i)=\sigma_{i}v_i^\pc$ so that (\ref{liminf}) holds.

We now choose $a_i$ and $b_i$ to ensure that $\fps$ is continuous, i.e. such that the line $y=b_iu -a_i$ passes through the points  $(u_{i+1},\sigma_{i+1}u_{i+1}^\pc )$ 
 and $(v_i,\sigma_{i}v_i^\pc )=(\theta u_{i+1},\sigma_{i}\theta^\pc   u_{i+1}^\pc )$. This yields
\be
 b_i=\frac{u_{i+1}^{\pc -1}(\theta^{\pc}\sigma_{i}-\sigma_{i+1})}{\theta -1}>0,\qquad a_i=\frac{u_{i+1}^{\pc }(\theta^{\pc}\sigma_{i}-\theta \sigma_{i+1})}{\theta -1}>0 .\label{anbn}\ee
 By construction $\fps$ is also increasing and Lipschitz on $[0,\delta ]$. In order that $\fps$ be convex on $[0,\delta ]$ we  require that
 \bs
 \pc \sigma_{i+1}u_{i+1}^{\pc -1}\le b_i\le \pc \sigma_{i}v_{i}^{\pc -1},
 \es
(by comparing the gradient of $f$ at the endpoints of the intervals), or equivalently
 \be
\theta^{\pc -1}( \theta -\pc ( \theta-1))\le  \frac{\sigma_{i+1}}{\sigma_{i}}\le \frac{\theta^\pc}{1+\pc (\theta^\pc -1)}.\label{convex}
 \ee
By  (\ref{params}), and since  $\sigma_{i+1}/\sigma_{i}\to 0$ as $i\to\infty$,  (\ref{convex}) holds for all $i$ sufficiently large.

Thus, $\fps$ is increasing, convex and Lipschitz on $[0,\delta ]$. It is clear that the domain of $\fps$ can then be extended to $[0,\infty )$ while still preserving
monotonicity, convexity and  Lipschitz continuity and also such that {\bf  (B)} holds.

\vspace{1cm}

{\it Step 3}. Next we show that $\fps$ satisfies the remaining  hypotheses of Corollary~\ref{cor:BUPpde}. By Remark~\ref{rem:BUP}(i) it suffices to show that 
\bi
\item[(i)] $\fps (u)/u^p $ is non-decreasing on $(0,\delta )$, and 
\item[(ii)] $\displaystyle{\int_{0^+}\frac{\fps (u)}{u^{2+\alpha /n}} \, \dee u=\infty} $. 
\ei

For (i) let $\fs (u):=\fps (u)/u^p $. This {\it continuous}, piecewise differentiable  function is given explicitly by
\bs
\fs (u)=\left\{
\begin{array}{ll}
b_iu^{1 -p }-a_iu^{-p } ,  & u\in J_i, \\
\sigma_{i}u^{\pc -p }, & u\in M_i. 
\end{array}\right.
\es
Clearly $\fs$ is non-decreasing  on $M_i$ for all $i$.  On $J_i= [u_{i+1},v_i)$ we  have that
\bs
\fs\, ' (u)=u^{-p-1}(pa_{i}-(p-1)b_iu).
\es 
Hence $\fs\, ' \ge 0$ on $J_i$ if and only if 
\be
pa_{i}\ge (p-1)b_iv_i.\label{deriv}
\ee
Now, recalling (\ref{anbn}), we have
\bs
\frac{pa_{i}}{ (p-1)b_iv_i}&=& \frac{p(\theta^{\pc}\sigma_{i}-\theta \sigma_{i+1})}{\theta (p-1)(\theta^{\pc}\sigma_{i}-\sigma_{i+1})}\\
&=&  \frac{p(\theta^{\pc}-\theta \sigma_{i+1}/\sigma_{i})}{\theta (p-1)(\theta^{\pc}-\sigma_{i+1}/\sigma_{i})}
\to   \frac{p}{\theta (p-1)}
\es
as $i\to\infty$. Hence, by (\ref{params}),  (\ref{deriv}) holds for all $i$ sufficiently large. Thus $\fs$ is non-decreasing on $(0,\delta )$.

For (ii),
\bs
\int_{0}^\delta \frac{\fps (u)}{u^{2+\alpha /n}} \, \dee u & \ge &  \sum_{i=1}^\infty\int_{M_i}\frac{\fps (u)}{u^{2+\alpha /n}}\, \dee u=
\sum_{i=1}^\infty\int_{v_i}^{u_i}\frac{\sigma_{i}}{u}\, \dee u\\
& = & \sum_{i=1}^\infty \sigma_{i}\log ({u_i}/{v_i})= \sum_{i=1}^\infty \sigma_{i}\log \left({u_i}/{(\theta u_{i+1}})\right)\\
& = & \sum_{i=1}^\infty (e^{2i+1}-1)- \log\theta\sum_{i=1}^\infty e^{-i^2} \\
& =& \infty ,
\es
recalling (\ref{seq}).

{\it Step 4}. Finally we show that $\fps$  fails to satisfy assumption \cite[(B3)]{KST} when $\alpha =2$. In fact we establish a more general result:  for any $\alpha\in (0,2]$,   
we find a sequence $\lam_i\to 0$ such that 
$$\lim_{i\to\infty}\frac{\fps (\lam_i^2)}{\lam_i^\pc \fps(\lam_i)}=0.$$
Consequently   (\ref{diag}) fails in the special case  $\alpha =2$.

To achieve this we show that there is a sequence $\lam_i\in M_i=[v_i,u_i)$ such that  $\lam_i^2\in M_{i+1}$. It will then follow that
$$\lim_{i\to\infty}\frac{\fps (\lam_i^2)}{\lam_i^\pc \fps(\lam_i)}=\lim_{i\to\infty}\frac{\sigma_{i+1}(\lam_i^2)^\pc }{\lam_i^\pc (\sigma_{i}\lam_i^\pc )}
=\lim_{i\to\infty}\frac{\sigma_{i+1}}{\sigma_{i}}=0,$$
recalling that $\sigma_{i}=e^{-i^2}$.

Fix $1/2<q<1$ and let $\lam_i=v_i^q$. Clearly $\lam_i> v_i$ since $v_i<1$ and $q<1$.  It is also easily verified that  $\lam_i=\theta^q u_{i+1}^q<u_i$ for $i$ sufficiently large, recalling (\ref{seq}). Hence  $\lam_i\in M_i$ for such $i$.   
Next, 
\bs
\lam_i^2=v_i^{2q}=\theta^{2q} u_{i+1}^{2q}<u_{i+1}
\es 
for $i$ sufficiently large, since $2q>1$. Also, $\lam_i^2=v_i^{2q}> u_{i+1}^{2q}$ and $v_{i+1}=\theta u_{i+2}$. Hence in order to show that 
 $\lam_i^2>v_{i+1}$, it suffices to show that  $u_{i+1}^{2q}>\theta u_{i+2}$. This is readily verified for large $i$, recalling (\ref{seq}).
  It follows that $\lam_i^2\in M_{i+1}$, as required.


\begin{remark}
Consider the case $\alpha =2$. By taking $g(\mu )=\mu ^{p_F}$, any function $f$ satisfying (\ref{B3}) necessarily satisfies our condition {\bf  (S)}. Hence any convex function $f$ satisfying 
 the hypotheses of \cite[Theorem 4.1]{KST} also satisfies those of Corollary~\ref{cor:BUPpde}. 
Our distinguishing example therefore shows that, {\it  within the class of convex source terms},  Corollary~\ref{cor:BUPpde} is {\it strictly} stronger than \cite[Theorem 4.1]{KST}.
\label{rem:stronger} 
\end{remark}

\begin{remark}
It is reasonable to speculate whether the analogous condition to (\ref{B3}), with the power law $\mu^{p_F}$ replaced by  $\mu^\pc$, might provide the basis
for similar results to those in  \cite{KST}   for the fractional diffusion case $0<\alpha <2$. However, the $f$ constructed above satisfies 
\bs
\liminf_{\lam\to 0}\frac{f (\lam^2)}{\lam^{\pc} f(\lam)}=0
\es
for any $\alpha \in (0,2]$. Consequently, the $f$ constructed above pre-empts any improvemts that might possibly be obtained in this way, at least within the class of { convex} source terms.
 \end{remark}



\section{Global Existence}

In this section we consider the issue of global continuation of  locally bounded  solutions of (\ref{nhe}). We set  $Q_T= \R^n\times (0,T)$ and   write $\| \cdot \|_q$ for the norm in $L^q(\R^n  )$.

\begin{definition}
Let $T>0$. We say that a non-negative, measurable, finite almost everywhere function $w\colon Q_T\to\R$ is an integral supersolution  of (\ref{nhe}) on $Q_T$ if $w$ satisfies ${\mathscr F}(w;\phi)\le w$ almost everywhere in $Q_T$, with $\F$  as in (\ref{eq:VoC}).
\label{def:super}
\end{definition}



We recall the following well-known smoothing estimate for the fractional heat semigroup for $1\le q\le r\le\infty$ and $\phi\in L^q(\R^n )$ (see e.g., \cite[Lemma 3.1]{MYZ}):
\be\label{eq:smoothing}
  \left\|S_{\alpha}(t)\phi \right\|_{r}\le Ct^{-\frac{n}{\alpha}\left(\frac{1}{q} -\frac{1}{r} \right)}\|\phi\|_q,\qquad t>0,
 \ee
where $C=C(n,\A ,q,r)$.


For $f$ satisfying {\bf  (M)} we define the  non-decreasing function  $\ell \colon (0,\infty )\to [0,\infty )$ by
\be
\ell (u)=\sup_{ 0<s\le u}\frac{f(s)}{s}.
\label{eq:ell}
\ee

\begin{theorem}
Suppose $\phi\in L^1(\R^n  )\cap L^\infty(\R^n  )$,  $\phi\ge 0$ and $f$ satisfies {\bf  (M)}.  Let  $u(t;\phi)$ denote  the unique, non-negative solution
 of (\ref{nhe}) with maximal interval of existence $[0,T_{\phi})$. If
\be
\int_{0^+}u^{-\pc }\ell (u)\, \dee  u <\infty , \label{B2fin}
\ee
then  there exists  $\rho > 0$   such that for all  $\phi$ satisfying
$\|\phi\|_1+\|\phi\|_\infty \le \rho $ we have $T_{\phi}=\infty$ and
\be
0\le u(t;\phi)\le 2S_{\alpha}(t)\phi \label{eq:globalorder}
\ee
for all $t\ge 0$.
Consequently $\|u(t;\phi)\|_\infty\le 2 Ct^{-n/\alpha }\|\phi\|_1$ for all $t>0$, where $C=C(n,\A ,1 ,\infty)$.
\label{thm:global}
\end{theorem}

\begin{proof}  We will show that for  suitably small $\rho >0$,  $w\coloneqq 2S_{\alpha}(t)\phi $ is   an integral supersolution  of (\ref{nhe}) for all $t\ge 0$. Via
the  monotone iteration scheme $u_{n+1}=\F (u_n;\phi )$ we then obtain a decreasing sequence of  functions $u_n$ such that $0\le u_n\le w$ and converging to a solution 
$\tilde{u}(t;\phi)$ of (\ref{nhe}). See, for example,    \cite{LRSV,RS2} for the case $\alpha =2$ and  \cite{Li} for the fractional case for general results of this kind.
 By standard uniqueness results 
we may then  deduce that $\tilde{u}(t;\phi)=u(t;\phi)$ and $0\le u(t;\phi)\le w$, yielding (\ref{eq:globalorder}). The $L^\infty$-bound for $u$ then follows by   $L^1$-$L^\infty$ smoothing.

First set $C_1=C(n,\A ,\infty ,\infty)$ and choose $\rho$ such that  $\rho C_1 \le 1$. Then  choose  $\tau >0$ such that $2\ell (2)\tau <1$. By (\ref{eq:smoothing}) 
 we have $\|S_{\alpha}(t)\phi\|_\infty\le C_1\|\phi\|_\infty\le C_1\rho \le 1$ for all $t> 0$. In particular, for all $t\in (0,\tau ]$ we have
\bs
{\mathscr F}(w;\phi )-w & = & S_{\alpha}(t)\phi+\int_0^t S_{\alpha}(t-s) f(w(s))\,\dee s-w\\
&\le &S_{\alpha}(t)\phi +\int_0^t S_{\alpha}(t-s) \left[\ell (w(s))w(s)\right]\,\dee s-w\\
&=& \int_0^t S_{\alpha}(t-s)\left[\ell \left(2S_{\alpha}(s)\phi \right)2S_{\alpha}(s)\phi \right]\,\dee s -S_{\alpha}(t)\phi \\
&\le & \int_0^t S_{\alpha}(t-s)\left[\ell \left(\|2S_{\alpha}(s)\phi \|_{\infty}\right)2S_{\alpha}(s)\phi \right]\,\dee s -S_{\alpha}(t)\phi \\
&\le & \int_0^t S_{\alpha}(t-s)\left[2\ell (2)S_{\alpha}(s)\phi \right]\,\dee s-S_{\alpha}(t)\phi \\
&=& 2\ell (2)\int_0^t S_{\alpha}(t)\phi \,\dee s-S_{\alpha}(t)\phi \\
&\le &\left(2\ell (2)\tau -1 \right)S_{\alpha}(t)\phi \\
&\le & 0.
\es
For  $t> \tau $ we proceed as above, utilizing the  $L^1$-$L^\infty$ smoothing estimate $\|S_{\alpha}(t)\phi \|_\infty\le C_2t^{-n/\alpha }\|\phi \|_1\le C_2\rho t^{-n/\alpha }$, where
$C_2:=C(n,\A ,1 ,\infty)$. Whence, 
\bs
&&{\mathscr F}(w;\phi )-w= S_{\alpha}(t)\phi+\int_0^\tau  S_{\alpha}(t-s) f(w(s))\,\dee s+\int_\tau ^t S_{\alpha}(t-s) f(w(s))\,\dee s-w\\
&\le&\left(2\ell (2)\tau -1\right)S_{\alpha}(t)\phi +\int_\tau ^t S_{\alpha}(t-s) \left[\ell (w(s))w(s)\right]\,\dee s\\
&\le& \left(2\ell (2)\tau -1\right)S_{\alpha}(t)\phi +\int_\tau ^t S_{\alpha}(t-s)\left[\ell \left(2S_{\alpha}(s)\phi \right)2S_{\alpha}(s)\phi \right]\,\dee s\\
&\le& \left(2\ell (2)\tau -1\right)S_{\alpha}(t)\phi +\int_\tau ^t S_{\alpha}(t-s)\left[\ell \left(\|2S_{\alpha}(s)\phi \|_{\infty}\right)2S_{\alpha}(s)\phi \right]\,\dee s\\
&\le& \left(2\ell (2)\tau -1\right)S_{\alpha}(t)\phi +\int_\tau ^t S_{\alpha}(t-s)\left[\ell (2C_2\rho s^{-n/\alpha })2S_{\alpha}(s)\phi \right]\,\dee s\\
&=& \left(2\ell (2)\tau -1\right)S_{\alpha}(t)\phi +2S_{\alpha}(t)\phi \int_\tau ^t \ell (2C_2\rho  s^{-n/\alpha })\,\dee s\\
&=& \left(2\ell (2)\tau -1+({\alpha}/n)2^{\pc }(C_2\rho )^{\alpha /n}
\int_{2C_2\rho t^{-n/\alpha }}^{2C_2\rho \tau ^{-n/\alpha }} x^{-\pc }\ell (x)\,\dee x\right)S_{\alpha}(t)\phi \\
&\le& \left(2\ell (2)\tau -1+({\alpha}/n)2^{\pc }(C_2\rho )^{\alpha /n}\int_{0}^{2C_2\rho \tau ^{-n/\alpha }} x^{-\pc }\ell (x)\,\dee x\right)S_{\alpha}(t)\phi \\
& \le & 0
\es
for $\rho $ sufficiently small (and independently of $t$), by (\ref{B2fin}).
\end{proof}

We are now in a position to state our main  result.

\begin{theorem}[Blow-up Dichotomy] If $f$ satisfies {\bf  (M)},  {\bf  (C)}, {\bf  (B)} and {\bf  (S)}, then the following are equivalent:
\begin{enumerate}
  \item[(a)] the PDE (\ref{nhe}) has the blow-up property;
  \item[(b)] the ODE (\ref{ODE})  has the blow-up property;
  \item[(c)] $\displaystyle{\int_{0^+}\frac{f(u)}{u^{2+\alpha /n}}\, \dee  u =\infty .}$
\end{enumerate}
\label{thm:dich}
\end{theorem}

\begin{proof}
By the contrapositive of Theorem~\ref{thm:global}, (a) implies (c) (noting that $\ell (u)=f(u)/u$ for $f$ convex).
By Theorem~\ref{thm:BUPode}, (c) implies (b). 
By Theorem~\ref{thm:blowup}, (b) implies (a).
\end{proof}

\begin{example}
In the special case of the homogeneous power law nonlinearity
$f(u)=u^p$ with $p>1$, Theorem~\ref{thm:dich} shows that (\ref{nhe}) has the blow-up property if and only if $1<p\le \pc $, as is well known  \cite{Fuj66,Fuj70,Hay,Sug}.
\label{eg:1}
\end{example}


\begin{example}
For $\B >0$ take
\bs
f(u)=\left\{
\begin{array}{ll}
u^{\pc }\left(\log (1/u)\right)^{-\B}, & u\in (0,c_0 ),\\
\hat{f}(u), &  u\ge c_0,
\end{array}
\right.
\es
with $f(0)=0$ and $\hat{f}$ and $c_0>0$ to be specified. For  sufficiently small $c_0 >0$, $f$ can be shown to be convex (via a tedious calculation) and $f(u)/u^{p}$ 
is non-decreasing on $ (0,c_0 )$  for any $p \in (1,\pc ]$. By Remark~\ref{rem:BUP} (i) it follows that $f$ satisfies {\bf  (S)}.
Clearly $\hat{f}$ can then be chosen such that {\bf  (M)}, {\bf  (B)} and {\bf  (C)} hold.   Computation of the integral in Theorem~\ref{thm:dich} (c) then shows that (\ref{nhe}) 
has the blow-up property if and only if  $0<\B \le 1$. See the example in \cite[Section 5]{KST} for the special case  $\alpha =2$.
\label{eg:logzero}
\end{example}

\begin{remark}
Under the assumptions of Theorem~\ref{thm:dich} we could rewrite the equivalences as follows:
\begin{enumerate}
  \item[(a)] there exist positive, global solutions of  the PDE (\ref{nhe});
  \item[(b)] there exist positive, global solutions of  the ODE (\ref{ODE});
  \item[(c)] $\displaystyle{\int_{0^+}\frac{f(u)}{u^{2+\alpha /n}}\, \dee  u <\infty .}$
\end{enumerate}
\label{thm:globaldich}

\end{remark}



\section{Concluding remarks}

We have established  a new blow-up dichotomy for positive  solutions of fractional semilinear heat equations, extending
those of  \cite{Fuj66, Hay, KST, Sug}. In particular, for a class of convex nonlinearities we have established  an equivalence between  the PDE (\ref{nhe}) and the  ODE (\ref{ODE})
with respect to the blow-up property. Furthermore  we have determined necessary and sufficient conditions on the nonlinearity $f$, in the form of a non-integrability condition near zero,
 for this blow-up property to hold.

 When viewed in their integral form, there is an obvious {\em formal} similarity between the PDE  and an auxiliary ODE:  the PDE  (\ref{nhe}) 
can be cast as
\bs u(t)=S_{\alpha}(t)\phi+\int_0^t S_{\alpha}(t-s)f(u(s))\, \dee  s,\es
while the ODE (\ref{ODE}) can be written as 
\bs x(t) =(t/t_0)^{-n/\alpha }x(t_0) +\int_{t_0}^t   \left(t/s\right)^{-n/\alpha }f(x(s))\,\dee s.\es
The similarity arises 
 when considering   the decay rate of the  operator norm of $S_{\alpha}(t):L^1(\R^n )\to L^{\infty}(\R^n )$, which is given (via the smoothing estimate (\ref{eq:smoothing}))  by
\bs\| S_{\alpha}(t)\|_{L^1\to L^{\infty}}\le Ct^{-n/\alpha} .\es
It is  intriguing that this formal similarity manifests itself as an equivalence with respect to the blow-up property.

 It would be interesting to know whether the technical hypothesis {\bf  (S)} can be removed in Theorem~\ref{thm:BUPode} and consequently in Theorem~\ref{thm:dich}.
This would yield a sharper and perhaps  more natural result.  However, recalling Remark~\ref{rem:BUP} (i), we suspect that the stronger (but more easily verified) assumption that $f(u)/u^p$ be non-decreasing  near zero for some $p>1$, will prove more useful in applications.
 Similarly we would like to better understand the r\^ole of the convexity assumption on $f$. It is this convexity that permits us to show, via Jensen's inequality, that blow-up of the ODE implies blow-up of the PDE.  It remains open whether our blow-up equivalence result can be obtained without the convexity assumption and without assuming  {\bf  (S)}.
 
 Finally, we mention  that the analysis of fractional semilinear parabolic  equations such as  (\ref{nhe}) is intimately related to the study of symmetric $\alpha$-stable processes in probability theory. It seems reasonable to hope that   our work might have parallels in that domain  and provide new insights for such processes.

 \section*{Acknowledgements}
   RL was  partially supported by the QJMAM Fund for Applied Mathematics. MS was  partially supported by NCN grant 2017/26/D/ST1/00614.

\bibliographystyle{model1num-names}
\bibliography{<your-bib-database>}

\end{document}